\documentclass[a4paper]{amsart}

\usepackage[T1]{fontenc}
\usepackage{amsmath, amssymb, amsthm, mathrsfs, mathtools, marvosym, graphicx}
\usepackage{varioref}
\usepackage[backref=page]{hyperref}
\usepackage{enumitem, xspace, ifthen, comment}
\usepackage[all]{xy}
\usepackage[dvipsnames]{xcolor}
\usepackage[nameinlink, capitalize]{cleveref}

\setlist[enumerate]{
  label=(\thethm.\arabic*),
  before={\setcounter{enumi}{\value{equation}}},
  after={\setcounter{equation}{\value{enumi}}},
  itemsep=1ex
}

\setlist[itemize]{
  leftmargin=*,
  itemsep=1ex,
  label=$\circ$
}

\newcommand{\mypagesize}{
  \addtolength{\textwidth}{2pt}
  \addtolength{\textheight}{27pt}
  \calclayout
}

\mypagesize

\newtheorem*{thm-plain}{Theorem}
\newtheorem{thm}{Theorem}[section]
\newtheorem{lem}[thm]{Lemma}
\newtheorem{prp}[thm]{Proposition}
\newtheorem{cor}[thm]{Corollary}
\newtheorem{fct}[thm]{Fact}

\numberwithin{equation}{section}

\theoremstyle{definition}
\newtheorem{dfn}[thm]{Definition}

\newtheorem*{dfn-plain}{Definition}

\theoremstyle{remark}
\newtheorem{clm}[thm]{Claim}
\crefname{clm}{Claim}{Claims}
\newtheorem{obs}[thm]{Observation}
\newtheorem{awlog}[thm]{Additional Assumption}
\crefname{awlog}{Assumption}{Assumptions}

\newtheorem{rem}[thm]{Remark}

\newtheorem*{rem-plain}{Remark}

\newenvironment{myequation}{%
  \setcounter{equation}{\value{thm}}%
  \equation
}{%
  \endequation%
  \stepcounter{thm}%
}

\newcommand{\inv}{^{-1}}
\newcommand{\from}{\colon}

\newcommand{\lto}{\longrightarrow}

\newcommand{\x}{\times}
\newcommand{\inj}{\hookrightarrow}

\newcommand{\isom}{\cong}
\newcommand{\defn}{\coloneqq}
\newcommand{\ndef}{\eqqcolon}
\newcommand{\tensor}{\otimes}

\newcommand{\wt}{\widetilde}

\renewcommand{\d}{\mathrm d}
\newcommand{\del}{\partial}

\newcommand{\dual}{^{\smash{\scalebox{.7}[1.4]{\rotatebox{90}{\textup\guilsinglleft}}}}}
\newcommand{\ddual}{^{\smash{\scalebox{.7}[1.4]{\rotatebox{90}{\textup\guilsinglleft} \hspace{-.5em} \rotatebox{90}{\textup\guilsinglleft}}}}}

\newcommand{\factor}[2]{\left. \raise 2pt\hbox{$#1$} \right/\hskip -2pt \raise -2pt\hbox{$#2$}}

\newdir{ ir}{{}*!/-5pt/@^{(}} 
\newdir{ il}{{}*!/-5pt/@_{(}} 

\newcommand{\set}[1]{\left\{ #1 \right\}}

\def\rd#1.{\lfloor{#1}\rfloor}
\def\rp#1.{\lceil{#1}\rceil}
\def\tw#1.{\langle{#1}\rangle}

\renewcommand{\O}[1]{\mathscr{O}_{#1}}
\newcommand{\Omegap}[2]{\Omega_{#1}^{#2}}
\newcommand{\Omegar}[2]{\Omega_{#1}^{[#2]}}
\newcommand{\T}[1]{\mathscr{T}_{#1}}
\newcommand{\can}[1]{\omega_{#1}}

\newcommand{\Sing}[1]{{#1}_{\mathrm{sg}}}

\newcommand{\cc}[2]{\mathrm{c}_{#1}(#2)}

\def\Hnought#1.#2.{\mathit{\Gamma} \!\left( #1, #2 \right)}
\def\HH#1.#2.#3.{\mathrm{H}^{#1} \!\left( #2, #3 \right)}
\def\euler#1.#2.{\chi \!\left( #1, #2 \right)}
\def\HHbig#1.#2.#3.{\mathrm{H}^{#1} \!\big( #2, #3 \big)}
\def\hh#1.#2.#3.{h^{#1} \!\left( #2, #3 \right)}
\def\RR#1.#2.#3.{R^{#1} #2_* #3}
\def\HHc#1.#2.#3.{\mathrm{H}_{\mathrm{c}}^{#1} \!\left( #2, #3 \right)}
\def\Hh#1.#2.#3.{\mathrm{H}_{#1} \!\left( #2, #3 \right)}
\def\Hom#1.#2.{\mathrm{Hom} \!\left( #1, #2 \right)}
\def\sHom#1.#2.{\mathscr{H}\!om \!\left( #1, #2 \right)}
\def\Ext#1.#2.#3.{\mathrm{Ext}^{#1} \!\left( #2, #3 \right)}
\def\sExt#1.#2.#3.{\mathscr{E}\!xt^{#1} \!\left( #2, #3 \right)}

\newcommand{\PP}[1]{\mathbb P^{#1}}

\newcommand{\kahler}{K{\"{a}}hler\xspace}

\DeclareMathOperator{\Pic}{Pic}

\DeclareMathOperator{\Exc}{Exc}

\DeclareMathOperator{\supp}{supp}

\newcommand{\pg}[2]{p_g(#1, #2)}

\newcommand{\germ}[2]{\left( #2, #1 \right)} 

\renewcommand{\theta}{\vartheta}
\renewcommand{\phi}{\varphi}

\newcommand{\Q}{\ensuremath{\mathbb Q}}
\newcommand{\R}{\ensuremath{\mathbb R}}
\newcommand{\C}{\ensuremath{\mathbb C}}

\newcommand{\frE}{\mathfrak E}

\renewcommand{\frm}{\mathfrak m}

 \newcommand{\sE}{\mathscr E} \newcommand{\sF}{\mathscr F}
  
 \newcommand{\sK}{\mathscr K} \newcommand{\sL}{\mathscr L}

\definecolor{forrest}{RGB}{81,133,49}
\definecolor{mydarkblue}{RGB}{10,92,153}

\newcommand{\PreprintAndPublication}[2]{#1}

\title[The Lipman--Zariski conjecture in genus one higher]{The Lipman--Zariski conjecture \\ in genus one higher}

\author{Hannah Bergner}
\address{Mathematisches Institut, Albert-Ludwigs-Universit\"at Freiburg, Ernst-Zermelo-Stra\ss e~1, 79104 Freiburg im Breisgau, Germany}
\email{\href{mailto:Hannah.Bergner-c9q@ruhr-uni-bochum.de}{Hannah.Bergner-c9q@ruhr-uni-bochum.de}}
\urladdr{\href{http://home.mathematik.uni-freiburg.de/bergner/}{home.mathematik.uni-freiburg.de/bergner/}}

\author{Patrick Graf}
\address{Department of Mathematics, University of Utah, 155~South 1400~East, Salt Lake City, UT 84112}
\email{\href{mailto:patrick.graf@uni-bayreuth.de}{patrick.graf@uni-bayreuth.de}}
\urladdr{\href{http://www.pgraf.uni-bayreuth.de/en/}{www.graficland.uni-bayreuth.de}}

\date{May 6, 2020}
\thanks{The second author was supported in full by a DFG Research Fellowship.
The published version of this preprint was funded by the DFG and the University of Bayreuth in the funding programme Open Access Publishing.}
\keywords{Lipman--Zariski conjecture, surface singularities, surfaces with generically nef tangent sheaf, twisted vector fields}
\subjclass[2010]{14B05, 14J17, 32S25, 13N15}

\hypersetup{
  pdfauthor={Hannah Bergner, Patrick Graf},
  pdftitle={The Lipman-Zariski conjecture in genus one higher},
  pdfkeywords={Lipman-Zariski conjecture, surface singularities, surfaces with generically nef tangent sheaf, twisted vector fields},
  pdfstartview={Fit},
  pdfpagelayout={OneColumn},
  pdfpagemode={UseNone},
  linktoc=all,
  breaklinks,
  linkcolor=[RGB]{0 0 96},
  citecolor=[RGB]{96 0 0},
  urlcolor=[RGB]{0 96 0},
  colorlinks}

\begin{document}

\begin{abstract}
We prove the Lipman--Zariski conjecture for complex surface singularities with $p_g - g - b \le 2$.
Here $p_g$ is the geometric genus, $g$ is the sum of the genera of the exceptional curves and $b$ is the first Betti number of the dual graph.
This improves on a previous result of the second author.
As an application, we show that a compact complex surface with locally free tangent sheaf is smooth as soon as it admits two generically linearly independent twisted vector fields and its canonical sheaf has at most two global sections.
\end{abstract}

\maketitle

\thispagestyle{empty}

\section{Introduction}

The Lipman--Zariski conjecture asserts that a complex algebraic variety (or complex space) $X$ with locally free tangent sheaf $\T X$ is necessarily smooth.
Here $\T X = \mathscr{H}\!om_{\O X} \!\!\left( \Omegap X1, \O X \right)$ is the dual of the sheaf of \kahler differentials.
By the combined work of Lipman~\cite[Thm.~3]{Lip65}, Becker~\cite[Sec.~8, p.~519]{Becker78}, and Flenner~\cite[Corollary]{Flenner88}, it is known that \emph{it suffices to prove the conjecture for normal surface singularities.}

In a previous paper~\cite{LowGenusLZ1}, the second author dealt with surface singularities that are ``not too far'' from being rational.
To make this precise, recall that for a normal surface singularity $\germ0X$, the following invariants are defined in terms of (but not dependent on the choice of) a log resolution $f \from Y \to X$ with exceptional divisor $E = E_1 + \dots + E_r$ and dual graph $\Delta = \Delta(E)$:
\begin{align*}
p_g & \defn \dim_\C \left( \RR1.f.\O Y. \right)_0, && \hspace{-5em} \text{the (geometric) genus,} \\
g & \defn \sum_{i=1}^r \hh1.E_i.\O{E_i}., \\[.75ex]
b & \defn b_1(\Delta), && \hspace{-5em} \text{the first Betti number of $\Delta$.}
\end{align*}
In this notation, the main result of~\cite{LowGenusLZ1} (albeit formulated in a different way) is the confirmation of the Lipman--Zariski conjecture in the case $p_g - g - b \le 1$.
The purpose of this note is to push that result one step further, to $p_g - g - b \le 2$.
This also explains the title, which on its own is rather cryptic.

\begin{thm}[Lipman--Zariski conjecture in genus one higher] \label{lowgenus lz}
Let $\germ0X$ be a normal complex surface singularity, with invariants $p_g$, $g$, and $b$ as above.
Assume that $p_g - g - b \le 2$.
Then the Lipman--Zariski conjecture holds for $\germ0X$.
That is, if $\T X$ is free, then $\germ0X$ is smooth.
\end{thm}

\subsection*{Global Corollaries}

In~\cite{LowGenusLZ1}, the second author used his (local) main result to study compact complex surfaces whose tangent sheaf satisfies some \emph{global} triviality properties.
Naturally, our stronger \cref{lowgenus lz} also has new applications in this global setting.
First of all, the proof of~\cite[Cor.~1.4]{LowGenusLZ1} can be simplified to some extent.
For the reader's convenience, we repeat the statement here.

\begin{cor}[Surfaces with generically nef tangent sheaf] \label{TX gen nef}
Let $X$ be a complex-projective surface such that $\T X$ is locally free and generically nef.
Then $X$ is smooth.
\end{cor}

Recall that generic nefness of a vector bundle $\sE$ on a normal projective surface $X$ means the following:
there exists an ample line bundle $H$ on $X$ such that if $C \subset X$ is a general element of the linear system $|mH|$, for $m \gg 0$, then the restriction $\sE|_C$ is nef.

A second application concerns compact complex surfaces $X$ which are not necessarily \kahler.
By a \emph{twisted vector field} on $X$, we mean a global section of $\T X \tensor \sL$, where $\sL$ is a line bundle with vanishing real first Chern class $\cc1\sL \in \HH2.X.\R.$.

\begin{cor}[Surfaces with two twisted vector fields] \label{TX twisted}
Let $X$ be a compact complex surface such that $\T X$ is locally free.
Suppose that $X$ admits two twisted vector fields $v_i \in \HH0.X.\T X \tensor \sL_i.$, $i = 1, 2$, which are linearly independent at some point.
Assume furthermore that $\dim_\C \HH0.X.\can X. \le 2$.
Then $X$ is smooth.
\end{cor}

\noindent
This result generalizes~\cite[Cor.~1.2]{LowGenusLZ1}, where $X$ was assumed to be almost homogeneous.
Note that this is nothing but the special case where both $\sL_i \isom \O X$.

\begin{rem} \label{c1 zero}
The wedge product $v_1 \wedge v_2$ is a nonzero global section of $\can X \dual \tensor \sL_1 \tensor \sL_2$, multiplication by which gives an injection $\HH0.X.\can X. \inj \HH0.X.\sL_1 \tensor \sL_2.$.
Thus the assumption on the dimension of $\HH0.X.\can X.$ is automatically satisfied e.g.~if $X$ is \kahler or if $\sL_1 \tensor \sL_2 \isom \O X$.

However, on a non-\kahler surface, having vanishing first Chern class is a rather weak condition on a line bundle.
Indeed, a line bundle with $\mathrm c_1 = 0$ can have Kodaira dimension one (and hence arbitrarily many global sections).
The easiest example is probably given by a Hopf surface of algebraic dimension one.
A more interesting example would be a primary Kodaira surface, or more generally any elliptic fibre bundle $S \to C$ that is not topologically trivial.
In this case, $\HH2.S.\R.$ can be arbitrarily large (depending on $C$), but $\phi^* \from \HH2.C.\R. \to \HH2.S.\R.$ always is the zero map~\cite[Prop.~V.5.3]{BHPV04}.
\end{rem}

\begin{rem}
The proof of \cref{TX twisted} shows the following:
Assume that for some integer $C$, we knew the Lipman--Zariski conjecture for surface singularities satisfying $p_g - g - b \le C$.
Then the additional assumption in \cref{TX twisted} can be weakened to ``$\dim_\C \HH0.X.\can X. \le C$''.
\end{rem}

\subsection*{Acknowledgements}

We would like to thank St{\'{e}}phane Druel for pointing out to us the reference~\cite{Seidenberg67}.
The second author thanks the University of Utah for providing a most hospitable working environment.
We are also grateful to the ref\-eree for constructive criticism, in particular for suggesting a reformulation of the main result that is not only more concise, but also more general.

\section{Notation and basic facts}

The \emph{sheaf of \kahler differentials} of a reduced complex space $X$ is denoted $\Omegap X1$.
The \emph{tangent sheaf}, its dual, is denoted $\T X \defn \sHom \Omegap X1.\O X.$.
If $Z \subset X$ is a closed subset, then $\T X(-\log Z) \subset \T X$ denotes the subsheaf of vector fields tangent to $Z$ at every point of $Z$.
The \emph{canonical sheaf} of $X$ is denoted $\can X$.
If $X$ is normal, the \emph{sheaf of reflexive differential $1$-forms} is defined to be the double dual of $\Omegap X1$, or the dual of $\T X$.
We denote it by $\Omegar X1 \defn \left( \Omegap X1 \right) \ddual$.
It is isomorphic to $i_* \big( \Omegap{X^\circ}1 \big)$, where $i \from X^\circ \inj X$ is the inclusion of the smooth locus.
If $X$ is compact and $\sF$ is a coherent sheaf on $X$, we write $\hh i.X.\sF. \defn \dim_\C \HH i.X.\sF.$.

\begin{dfn}[Resolutions]
A \emph{resolution of singularities} of a reduced complex space $X$ is a proper bimeromorphic morphism $f \from Y \to X$, where $Y$ is smooth.
We say that the resolution is \emph{projective} if $f$ is a projective morphism.
A \emph{log resolution} is a resolution whose exceptional locus $E = \Exc(f)$ is a simple normal crossings divisor, i.e.~a normal crossings divisor with smooth components.
A resolution is said to be \emph{strong} if it is an isomorphism over the smooth locus of $X$.
\end{dfn}

\begin{fct}[Functorial resolutions] \label{functorial res}
Let $X$ be a normal complex space.
Then there exists a strong log resolution $f \from Y \to X$ projective over compact subsets, called the \emph{functorial resolution}, such that $f_* \T Y(-\log E)$ is reflexive.
This means that for any vector field $\xi \in \Hnought U.\T X.$, $U \subset X$ open, there is a unique vector field
\[ \wt\xi \in \Hnought f\inv(U).\T Y(-\log E). \]
which agrees with $\xi$ wherever $f$ is an isomorphism.
\end{fct}

\cref{functorial res} is proven in~\cite[Thms.~3.36 and~3.45]{Kol07}, but concerning the reflexivity of $f_* \T Y(-\log E)$ see also~\cite[Thm.~4.2]{GK13}.
If $X$ is a surface, the functorial resolution is the same as the minimal good resolution.

\begin{dfn}[Geometric genus] \label{pg}
Let $\germ0X$ be a normal surface singularity, and let $f \from Y \to X$ be an arbitrary resolution.
The \emph{(geometric) genus} $\pg X0$ is defined to be the dimension of the stalk $(\RR1.f.\O Y.)_0$.
Alternatively, choosing the representative $X$ of the germ $\germ0X$ to be Stein, we may set $\pg X0 \defn \dim_\C \HH1.Y.\O Y.$.
This definition is independent of the choice of $f$.
\end{dfn}

The following statement can be found in~\cite[Thm.~5]{Seidenberg67}, in slightly greater generality and with an algebraic proof.
Another reference is~\cite[proof of Prop.~1.2]{BurnsWahl74}.
We include our own proof, which is more geometric in spirit.

\begin{prp}[Derivations in the presence of an isolated singularity] \label{derivations}
Let $\germ0X$ be a normal isolated singularity \emph{which is not smooth}.
Then every \C-linear derivation $\delta \from \O{X,0} \to \O{X,0}$ factors through the maximal ideal $\frm_0 \subset \O{X,0}$.
In other words, $\delta(\O{X,0}) \subset \frm_0$.
\end{prp}

In geometric terms, this says that ``every vector field vanishes at the singular point'' or more generally, ``every vector field is tangent to the singular locus''.

\begin{proof}[Proof of \cref{derivations}]
We use the correspondence between derivations, vector fields and local \C-actions as described in~\cite[\S 1.4, \S 1.5]{Akh95}.
Let $\delta$ be a derivation of $\O{X,0}$.
We have an induced local \C-action $\Phi \from \C \x X \to X$.
By the definition of local group action, $\Phi(t, -)$ is an automorphism of the germ $\germ0X$ for every (sufficiently small) $t \in \C$.
Since $0 \in X$ is the unique singular point of $X$, it follows that $\Phi(t, 0) = 0$ for every $t \in \C$.
In other words, the singular point is fixed by the action $\Phi$.
Now, we can recover $\delta$ from $\Phi$ by the formula
\begin{myequation} \label{kaup}
\delta(f)(x) = \frac\d{\d t}\bigg|_{t=0} f \big( \Phi(t, x) \big)
\end{myequation}
for every $f \in \O{X,0}$.
Plugging the statement about the singular point being fixed into~\labelcref{kaup}, we arrive at $\delta(f)(0) = 0$ for every function germ $f$.
Hence $\delta(\O{X,0}) \subset \frm_0$, as desired.
\end{proof}

Finally, we rely crucially on the following Hodge-theoretic result by van Straten and Steenbrink.

\begin{fct}[\protect{\cite[Cor.~1.4]{SvS85}}] \label{SvS}
Let $\germ0X$ be a normal surface singularity and $f \from Y \to X$ a log resolution with reduced exceptional divisor $E \subset Y$.
Then the map
\[ \factor{\Omegar X1}{f_* \Omegap Y1} \xrightarrow{\quad\d\quad} \factor{\can X}{f_* \can Y(E)} \]
induced by exterior derivative is injective.
\end{fct}

\section{Proof of \cref{lowgenus lz}}

Let $\set{ v_1, v_2 }$ be a local basis of $\T X$ and let $\set{ \alpha_1, \alpha_2 }$ be the dual basis of $\Omegar X1$, defined by $\alpha_i(v_j) = \delta_{ij}$.
Furthermore, let $f \from Y \to X$ be the functorial resolution of $X$ and $E \subset Y$ its exceptional locus.
We isolate the following observation from the proof of~\cite[Thm.~1.1]{LowGenusLZ1}, to which we also refer for more details.

\begin{obs} \label{ext obs}
If the basis $\set{ \alpha_1, \alpha_2 }$ can be chosen in such a way that say $\d\alpha_2 \in f_* \can Y(E)$, that is, $f^*(\d\alpha_2)$ has at most simple poles along $E$, then $\germ0X$ is smooth.
\end{obs}

\begin{proof}[Sketch of proof]
By \cref{SvS}, we see that $\alpha_2 \in f_* \Omegap Y1$, that is, $f^* \alpha_2$ extends to a holomorphic $1$-form $\wt\alpha_2$ on $Y$.
On the other hand, $v_2$ extends to a holomorphic vector field $\wt v_2$ on $Y$ tangent to $E$, by \cref{functorial res}.
As $\wt\alpha_2(\wt v_2)$ is identically one, $\wt v_2$ cannot have any zeroes.
It follows that $E$, if non-empty, consists of a single smooth elliptic curve.
Hence $\germ0X$ is log canonical and we may apply~\cite[Cor.~1.3]{GK13}.
(We could also appeal to the argument in~\cite[(1.6)]{SvS85}, or in fact even do this case completely by hand.)
\end{proof}

\begin{clm} \label{dim estimate}
We have $\dim \factor{\can X}{f_* \can Y(E)} = p_g - g - b$.
\end{clm}

\begin{proof}
Consider the short exact sequence
\[ 0 \lto \underbrace{\factor{f_* \can Y(E)}{f_* \can Y}}_{\ndef \sK} \lto \factor{\can X}{f_* \can Y} \lto \factor{\can X}{f_* \can Y(E)} \lto 0. \]
The middle term has dimension exactly $p_g$ by~\cite[Prop.~4.45(6)]{KM98}.
Hence it suffices to show that $\dim \sK = g + b$.
To this end, consider the residue sequence
\[ 0 \lto \can Y \lto \can Y(E) \lto \can E \lto 0. \]
Since $\RR1.f.\can Y. = 0$ by Grauert--Riemenschneider vanishing~\cite[Thm.~2.20.1]{Kol07}, and since $E$ is Cohen--Macaulay, we get $\dim \sK = \hh0.E.\can E. = \hh1.E.\O E.$.
A standard computation on the normalization of $E$ yields
\[ \hh1.E.\O E. = g + |\Sing E| - r + 1. \]
In terms of the dual graph $\Delta = \Delta(E)$, clearly $r$ is the number of vertices and $|\Sing E|$ is the number of edges.
But it is a general fact\footnote{Proof: Let $T \subset G$ be a maximal subtree. Then $T$ has exactly $r - 1$ edges. The map $G \to \factor GT$ is a homotopy equivalence, and $\factor GT$ is a wedge sum of $n - (r - 1)$ circles.} that the first Betti number of a (connected, undirected) graph $G$ with $r$ vertices and $n$ edges is $n - r + 1$.
So $\hh1.E.\O E. = g + b$, as desired.
\end{proof}

\begin{clm} \label{gorenstein}
The $2$-form $\sigma \defn \alpha_1 \wedge \alpha_2$ is a generator of $\can X$.
In particular, $X$ is Gorenstein.
\end{clm}

\begin{proof}
Define a map $\O X \to \can X$ by sending $1 \mapsto \sigma$.
This is an isomorphism on the smooth locus $X \setminus \set0$.
Then it is an isomorphism everywhere, as $X$ is normal and the sheaves $\O X$ and $\can X$ are reflexive.
\end{proof}

By \cref{gorenstein}, every element in $\factor{\can X}{f_*\can Y(E)}$ can be written as (the class of) $\rho \cdot \sigma$ for some holomorphic function germ $\rho \in \O{X,0}$.
If $\frm \defn \frm_0 \subset \O{X,0}$ is the maximal ideal, consider the linear subspace
\[ \factor{\frm\,\can X}{f_* \can Y(E)}
  = \set{ \, \rho \cdot \sigma \;|\; \rho(0) = 0 \, }
  \subset \factor{\can X}{f_* \can Y(E)}. \]
Unless $\factor{\can X}{f_* \can Y(E)} = 0$, this subspace has codimension one.
In any case, it has dimension $\le 1$ by \cref{dim estimate} and the assumption $p_g - g - b \le 2$.
(This is the only place where that assumption is used.)
Thus if the images of $\d\alpha_1$ and $\d\alpha_2$ are both contained in $\factor{\frm\,\can X}{f_* \can Y(E)}$, they are linearly dependent, say $\d\alpha_1 + \lambda \cdot \d\alpha_2 = 0$ for some $\lambda \in \C$.
Considering the basis $\set{ \alpha_1 + \lambda \alpha_2, \alpha_2 }$ of $\Omegar X1$, we can apply \cref{ext obs} to conclude that $\germ0X$ is smooth.
After possibly interchanging $\alpha_1$ and~$\alpha_2$, we may hence without loss of generality make the following

\begin{awlog} \label{one}
We have $\d\alpha_1 \not\in \frm\,\can X$.
\end{awlog}

Writing $\d\alpha_j = \rho_j \cdot \sigma$ for suitable $\rho_j \in \O{X,0}$, we thus have that $\rho_1 \not\in \frm$ is a unit.
So replacing $\alpha_2$ by $\rho_1 \alpha_2$ does not destroy the property of $\set{ \alpha_1, \alpha_2 }$ being a basis of~$\Omegar X1$.
After this replacement, $\d\alpha_1 = \sigma$.
Furthermore, note that
\[ \d \big( \underbrace{\alpha_2 - \rho_2(0) \alpha_1}_{\ndef \alpha_2'} \big) = \big( \underbrace{\rho_2 - \rho_2(0)}_{\in \frm} \big) \cdot \sigma, \]
and that we may replace $\alpha_2$ by $\alpha_2'$, again without destroying the basis property.
Summing up, this leads to the following simplification of our setting.

\begin{awlog} \label{two}
We have that $\d\alpha_1 = \sigma$ and $\d\alpha_2 \in \frm\,\can X$.
In other words, $\rho_1 \equiv 1$ and $\rho_2 \in \frm$.
\end{awlog}

We will also assume from now on that $\germ0X$ is not smooth, as otherwise there is nothing to prove.
Consider the $1$-form $\rho_2\alpha_1$.
A short calculation shows that $\d(\rho_2 \alpha_1) = \big( \rho_2 - v_2(\rho_2) \big) \cdot \sigma$, which by \cref{derivations} and \cref{two} defines an element in the at most one-dimensional vector space $\factor{\frm\,\can X}{f_* \can Y(E)}$.
If that element is nonzero, then there is a constant $\lambda \in \C$ with
\[ \d( \alpha_2 + \lambda \rho_2 \alpha_1 ) = \d \alpha_2 + \lambda \d( \rho_2 \alpha_1 ) = 0 \in \factor{\frm\,\can X}{f_* \can Y(E)} \subset \factor{\can X}{f_* \can Y(E)}. \]
Because $\set{ \alpha_1, \alpha_2 + \lambda \rho_2 \alpha_1 }$ is a basis of $\Omegar X1$, we can again apply \cref{ext obs} and then we are done.
Hence we may without loss of generality impose the following

\begin{awlog} \label{three}
We have $\d(\rho_2 \alpha_1) \in f_* \can Y(E)$.
\end{awlog}

For any function/differential form/vector field on $X$, we denote its lift to~$Y$ as a holomorphic or meromorphic object by a tilde.
Thus we have, for example, $\d \wt\alpha_1 = \wt\sigma = \wt\alpha_1 \wedge \wt\alpha_2$ and $\d\wt\alpha_2 = \wt\rho_2 \cdot \wt\sigma$.
Furthermore we let $\frE$ be the set of irreducible components of $E$ and we put
\begin{align*}
\frE^{\le 1} & \defn \set{ \, P \in \frE \;|\; \text{$\wt v_2$ vanishes to order at most $1$ along $P$} \, }, \\
\frE^{\ge 2} & \defn \set{ \, P \in \frE \;|\; \text{$\wt v_2$ vanishes to order at least $2$ along $P$} \, }.
\end{align*}
The order of vanishing, of course, refers to the largest integer $k$ such that locally near a general point of $P$, we can write $\wt v_2 = w^k v'$, where $w$ is a local defining equation for $P$ and $v'$ is a holomorphic vector field.

Let $P \in \frE^{\ge 2}$ be arbitrary, pick a point $p \in P$ not contained in any other component of $E$, and choose local holomorphic coordinates $z, w$ around $p$ such that locally $P = \set{ w = 0 }$.
We can then write $\wt v_2 = w^2 v'$ for some holomorphic vector field $v'$ defined near~$p$.
The $1$-form $\alpha_2 - \rho_2 \alpha_1$ satisfies
\[ \d( \alpha_2 - \rho_2 \alpha_1 ) = \big[ \rho_2 - \big( \rho_2 - v_2(\rho_2) \big) \big] \cdot \sigma
 = v_2(\rho_2) \cdot \sigma, \]
and the order of vanishing of $f^* \big( v_2(\rho_2) \big) = \wt v_2(\wt\rho_2) = w^2 v'(\wt\rho_2)$ along $P$ is strictly larger than the vanishing order of $\wt\rho_2$ along $P$.
Hence after replacing $\alpha_2$ by $\alpha_2 - \rho_2 \alpha_1$ finitely often, the $2$-form $\d\wt\alpha_2 = \wt\rho_2 \cdot \wt\sigma$ will be holomorphic at a general point of $P$.
This argument applies simultaneously to all $P \in \frE^{\ge 2}$ and we arrive at the

\begin{awlog} \label{four}
The $2$-form $\d\wt\alpha_2$ does not have a pole along any exceptional curve $P \in \frE^{\ge 2}$.
\end{awlog}

The next claim analogously deals with $\frE^{\le 1}$.
We stress that its proof relies only on \cref{three}, but not on \cref{four}.

\begin{clm} \label{five}
Along any curve $P \in \frE^{\le 1}$, the form $\d\wt\alpha_2$ has at worst a simple pole.
\end{clm}

\begin{proof}[Proof of \cref{five}]
Pick a component $P \in \frE^{\le 1}$ and let $p \in P$ and $z, w$ be as before.
There are holomorphic functions $a$ and $b$ near $p$ such that locally $\wt v_2 = a \frac\del{\del z} + b \frac\del{\del w}$.
Using Taylor expansion, we may write
\begin{align*}
a(z,w) & = a_0(z) + a_1(z) w + \cdots \qquad \text{and} \\
b(z,w) & = b_0(z) + b_1(z) w + \cdots,
\end{align*}
where the dots stand for terms of order at least $2$ in $w$ and $a_j$, $b_j$ are appropriate local holomorphic functions in one variable.
Since $\wt v_2$ is logarithmic with respect to $P = \set{ w = 0 }$, we in fact have $b_0 \equiv 0$.
As $P \in \frE^{\le 1}$, not all of $a_0$, $a_1$, $b_1$ can be identically zero.
\begin{itemize}
\item If $a_0 \not\equiv 0$, there is a point $q \in P$ near $p$ with $a_0(q) \ne 0$.
\item If $a_0 \equiv 0$, we may locally write $\wt v_2 = w \cdot v'$ with $v'$ holomorphic.
Since $a_1 \not\equiv 0$ or $b_1 \not\equiv 0$, there is a point $q \in P$ near $p$ with $v'(q) \ne 0$.
\end{itemize}
In both cases, locally near $q$ we have $\wt v_2 = h \cdot v'$ for a holomorphic function $h$ and a holomorphic vector field $v'$ with $v'(q) \ne 0$.
What is more, the function $h$ (which is either identically one, or equal to $w$) vanishes of order at most one along $P$.
Since $v'(q) \ne 0$, there exist local holomorphic coordinates $x, y$ near $q$ such that locally $v' = \frac\del{\del x}$ and thus $\wt v_2 = h \cdot \frac\del{\del x}$.

There are local meromorphic functions $g_{ij}$ such that with respect to the local coordinates $x, y$ we have $\wt\alpha_1 = g_{11} \, \d x + g_{12} \, \d y$ and $\wt\alpha_2 = g_{21} \, \d x + g_{22} \, \d y$.
Because $\wt\alpha_i(\wt v_2) = \delta_{i,2}$, we have in fact 
$\wt\alpha_1 = g_{12} \, \d y$ and $\wt\alpha_2 = h\inv \, \d x + g_{22} \, \d y$.
Thus by~\cref{two}
\begin{align*}
\phantom{\text{and} \qquad}
\d\wt\alpha_1 & = \wt\sigma = \wt\alpha_1 \wedge \wt\alpha_2 = -\frac{g_{12}}h \, \d x \wedge \d y \qquad \text{and} \\
\d\wt\alpha_2 & = \wt\rho_2 \cdot \wt\sigma = -\frac{\wt\rho_2 \, g_{12}}h \, \d x \wedge \d y.
\end{align*}
According to \cref{three} and \cref{SvS}, $f^*(\rho_2 \alpha_1) = \wt\rho_2 \, \wt\alpha_1 = \wt\rho_2 \, g_{12} \, \d y$ extends to a holomorphic $1$-form on $Y$.
In particular, $\wt\rho_2 \, g_{12} \, \d y$ has no pole along $P$ and therefore $\wt\rho_2 \, g_{12}$ is holomorphic.
This implies that $\d \wt\alpha_2 = - h\inv \, \wt\rho_2 \, g_{12} \, \d x \wedge \d y$ has at most a pole of order one along $P$, as desired.
\end{proof}

Taken together, \cref{four} and \cref{five} show that $\d\wt\alpha_2$ has at worst first order poles along any exceptional curve $P \in \frE = \frE^{\le 1} \cup \frE^{\ge 2}$.
In other words, we have $\d\alpha_2 \in f_* \can Y(E)$.
Applying once again \cref{ext obs}, we get that $(X,0)$ is smooth. \qed

\section{Proof of \cref{TX gen nef}}

By~\cite[Thm.~3]{Lip65}, $X$ is normal.
Let $f \from S \to X$ be the minimal resolution (i.e.~$K_S$ is $f$-nef).
We may assume that $X$ is not smooth.
Under this additional assumption, one shows as in the proof of~\cite[Claim~4.2]{LowGenusLZ1} that
\[ \hh0.X.{\RR1.f.\O S.}. = \sum_{x \in \Sing X} \pg Xx \le 2. \]
Hence every singular point of $X$ satisfies $p_g - g - b \le p_g \le 2$.
We conclude by \cref{lowgenus lz} that $X$ is smooth. \qed

\section{Proof of \cref{TX twisted}}

The following proposition, probably well-known to experts, will greatly simplify the proof.

\begin{prp}[Surfaces carrying a divisor homologous to zero] \label{surface hom zero}
Let $S$ be a smooth compact complex surface containing a nonzero effective divisor $D$ with
\[ \cc1D \defn \cc1{\O S(D)} = 0 \in \HH2.S.\R.. \]
Then either
\begin{enumerate}
\item the Kodaira dimension $\kappa(S) \le 0$, or
\item we have $\kappa(S) = 1$ and $\chi(\O S) = 0$.
\end{enumerate}
\end{prp}

\begin{proof}
Let $\pi \from S \to S_0$ be a minimal model, and set $D_0 \defn \pi_* D$.
Then $\cc1{D_0} = 0$ by \cref{c1 pushforward} below.
Furthermore $D_0 \ne 0$, as otherwise $D$ would be $\pi$-exceptional and hence $D^2 < 0$ by negative-definiteness of the intersection form~\cite[Thm.~III.2.1]{BHPV04}, contradicting the fact that $D^2 = \cc1D^2 = 0$.
Also $\chi(\O S)$ remains unchanged when passing to $S_0$.
We may thus assume that $S$ is minimal.

If $\kappa(S) = 2$, then $\mathrm c_1^2(S) > 0$ and $S$ is projective~\cite[Thm.~IV.6.2]{BHPV04}.
Thus the divisor $D$ cannot exist.
If $\kappa(S) = 1$, then the pluricanonical map $\phi \from S \to C$ is a relatively minimal elliptic fibration.
In this case
\begin{align*}
\chi(\O S) & = \deg_C \big( \RR1.\phi.\O S. \big) \dual && \text{\cite[Prop.~V.12.2]{BHPV04}} \\
 & = \deg_C \big( \phi_* \can{S/C} \big) && \text{\cite[Thm.~III.12.3]{BHPV04}} \\
 & \ge 0. && \text{\cite[Thm.~III.18.2]{BHPV04}}
\end{align*}
On the other hand, we have $\chi(\O S) \le 0$ by~\cite[Cor.~1.2]{CDP98}.
We conclude that $\chi(\O S) = 0$.
\end{proof}

\begin{lem} \label{c1 pushforward}
Let $S$ be a smooth compact complex surface and $\pi \from S \to S'$ the blowing-down of a $(-1)$-curve.
If $\sL \in \Pic(S)$ is a line bundle with $\cc1\sL = 0$, then so is $\sL' \defn (\pi_* \sL) \ddual$, where $(-) \ddual$ denotes the reflexive hull (or double dual) of a coherent sheaf.
\end{lem}

\begin{proof}
Being a reflexive rank~$1$ sheaf on a smooth surface, $\sL'$ is locally free.
Thanks to negative definiteness again~\cite[Thm.~III.2.1]{BHPV04}, we have $\sL = \pi^* \sL'$ and hence $\pi^* \big( \cc1{\sL'} \big) = \cc1\sL = 0$.
As $\pi^* \from \HH2.S'.\R. \to \HH2.S.\R.$ is injective~\cite[Thm.~I.9.1(iv)]{BHPV04}, it follows that $\cc1{\sL'} = 0$.
\end{proof}

\begin{lem} \label{RR}
Let $\sL$ be a line bundle on the smooth projective curve $C$ of genus $g$.
\begin{enumerate}
\item\label{RR.1} If $\deg \sL = 0$, then $\hh0.C.\sL. = \hh0.C.\sL\dual.$.
\item\label{RR.2} If $\deg \sL = 2g - 2$ but $\sL$ is not isomorphic to $\can C$, then $\hh0.C.\sL. = g - 1$.
\end{enumerate}
\end{lem}

\begin{proof}
\PreprintAndPublication{If $\sL$ is trivial, then both sides in~\labelcref{RR.1} equal $1$ and otherwise both sides are zero.
Turning to~\labelcref{RR.2}, we may write $\sL$ as $K_C - M$ with $M$ of degree zero but non-trivial.
Then $\hh0.C.\sL. = \hh1.C.M.$ by Serre duality.
But
\[ \hh1.C.M. = \hh0.C.M. - \chi(C, M) = - \chi(C, M) = g - 1 \]
by Riemann--Roch.}{
Well-known and hence left to the reader as an exercise.}
\end{proof}

We now turn to the proof of \cref{TX twisted}.
As in the previous corollary, we may assume that $X$ is normal, but not smooth.
Let $f \from S \to X$ be the functorial resolution and $\pi \from S \to S_0$ a run of the $K_S$-MMP.
\[ \xymatrix{
S \ar^-\pi[rr] \ar_-f[d] & & S_0 \\
X
} \]
By \cref{functorial res}, the twisted vector fields $v_i$ on $X$ lift to twisted vector fields $\wt v_i$ on~$S$.
These in turn can be pushed forward to twisted vector fields $v_i^0$ on $S_0$, by \cref{c1 pushforward}.
Furthermore, the Leray spectral sequence associated to $f_* \O S$ yields a five-term exact sequence
\begin{equation*}
\begin{aligned}
0 & \lto \HH1.X.\O X. \lto \HH1.S.\O S. \lto \HH0.X.{\RR1.f.\O S.}. \lto \\
  & \lto \HH2.X.\O X. \lto \HH2.S.\O S. \lto 0,
\end{aligned}
\end{equation*}
where the last map is Serre dual to $\HH0.S.\can S. \inj \HH0.X.\can X.$, and hence surjective.
We obtain an upper bound
\begin{myequation} \label{est}
\hh0.X.{\RR1.f.\O S.}. \le \hh1.S.\O S. + \hh0.X.\can X. - \hh0.S.\can S..
\end{myequation}

\begin{clm} \label{kahler}
Assume that $S$ has the property that every nonzero effective divisor $D \subset S$ satisfies $\cc1D \ne 0$.
(This applies in particular if $S$ is \kahler.)
Then $\hh0.X.\can X. \le 1$ and $\kappa(S) = -\infty$.
\end{clm}

\begin{proof}
If $\kappa(X, K_X) = -\infty$, then in particular $\hh0.X.\can X. = 0$ and also $\kappa(S) = -\infty$, since in any case $\kappa(S) \le \kappa(X, K_X)$.
We may thus assume that $\kappa(X, K_X) \ge 0$, i.e.~$|mK_X| \ne \emptyset$ for some $m \ge 1$.
Pick $D_m \in |mK_X|$, for a suitable $m$.
The wedge product of twisted vector fields $v_1 \wedge v_2$ is a nonzero global section of $\can X \dual \tensor \sL_1 \tensor \sL_2$.
Its zero divisor is thus an element $D_{-1} \in |{- K_X + L_1 + L_2}|$.
Then $D_m + mD_{-1} \in |m(L_1 + L_2)|$ is an effective divisor with first Chern class zero.
It follows that $D_m + mD_{-1} = 0$ (pull back along $f$ and use the assumption on $S$).
Hence $D_m = 0$, i.e.~$K_X$ is torsion and $\hh0.X.\can X. \le 1$.
If $f' \from S' \to X$ is the minimal resolution, we have
\[ K_{S'} = f^* K_X - E \sim_\Q -E \]
with $E \ge 0$ an effective $f'$-exceptional divisor.
If $E = 0$, then $X$ has canonical singularities, hence it is smooth~\cite[Cor.~1.3]{GK13}.
So $E \gneq 0$ and $\kappa(S) = \kappa(S') = \kappa(S', -E) = -\infty$.
\end{proof}

\begin{clm} \label{ruled}
If $\phi \from S_0 \to C$ is a ruled surface, then the genus $g(C) \le 1$.
\end{clm}

\begin{proof}
The vector fields $v_i^0$, being generically linearly independent, cannot both be tangent to the fibres of $\phi$.
Hence $\HH0.S_0.\phi^* \T C \tensor \sL_i. \ne 0$ for, say, $i = 1$.
This is the same as $\HH0.C.\T C \tensor \phi_* \sL_1.$ by the projection formula, so $\phi_* \sL_1 \ne 0$.
Since $\cc1{\sL_1} = 0$, $\sL_1$ must be trivial on the fibres of $\phi$.
Thus $\phi_* \sL_1$ is a line bundle and $\sL_1 = \phi^* (\phi_* \sL_1)$.
We conclude from this that $\deg_C \phi_* \sL_1 = 0$.
Summing up, the line bundle $\T C \tensor \phi_* \sL_1$ has a nonzero global section (the image of $v_1$) and its degree is $2 - 2g(C) + \deg_C \phi_* \sL_1 = 2 - 2g(C) \ge 0$.
This immediately implies the claim.
\end{proof}

\begin{clm} \label{kappa 1}
If $\kappa(S) = 1$ and $\phi \from S_0 \to C$ is the pluricanonical map, let $L$ be a divisor on $C$ corresponding to the line bundle $\phi_* \can{S_0/C}$.
Assume that $\deg L = 0$.
Then
\begin{itemize}
\item $\hh1.S.\O S. = g + \hh0.C.L.$, where $g$ is the genus of $C$, and
\item $\hh0.S.\can S. \ge \hh0.C.K_C + L.$.
\end{itemize}
\end{clm}

\begin{proof}
The Leray spectral sequence for $\phi$ and $\O{S_0}$ gives
\[ \hh1.S_0.\O{S_0}. = \hh1.C.\O C. + \hh0.C.{\RR1.\phi.\O{S_0}.}. = g + \hh0.C.L. \]
by~\labelcref{RR.1}.
On the other hand, if $m_1 F_1, \dots, m_k F_k$, $m_i \ge 2$, are the multiple fibres of $\phi$, then Kodaira's canonical bundle formula~\cite[Thm.~V.12.1]{BHPV04} reads
\[ K_{S_0} = \phi^* (K_C + L) + \sum_{i=1}^k (m_i - 1) F_i \ge \phi^* (K_C + L). \]
Taking global sections yields the second claim.
\end{proof}

Now by \cref{kahler}, either $\kappa(S) = -\infty$, or $\kappa(S) \in \set{0, 1}$ and $S$ contains a divisor with vanishing first Chern class.
By \cref{surface hom zero} and the Kodaira--Enriques classification~\cite[Table~10 on p.~244]{BHPV04}, we are left with the possibilities for $S_0$ listed in \cref{788} \vpageref{788}.

\begin{table}[!t]
\renewcommand{\arraystretch}{1.4}
\scalebox{.95}{%
\begin{tabular}{|l|l|l|l|}
\hline
$S_0$ & $\hh1.S.\O S.$ & $\hh0.X.\can X.$ & $\hh0.S.\can S.$ \\ \hline \hline

$\PP2$ or ruled          & $\le 1$ (\cref{ruled})         & $\le 1$ (\cref{kahler}) & $0$ \\ \hline
Class VII\textsubscript0 & $1$ \cite[Thm.~IV.2.7]{BHPV04} & $\le 2$                 & $0$ \\ \hline
Primary Kodaira          & $2$                            & $\le 2$                 & $1$ \\ \hline
Secondary Kodaira        & $1$                            & $\le 2$                 & $0$ \\ \hline
\parbox[t]{8.7em}{\raggedright Minimal properly elliptic, non-\kahler, and $\chi(\O S) = 0$ \vspace{.75ex}} & $g + h^0(L)$ & $\le 2$ & $\ge h^0(K_C + L)$ \\ \hline
\end{tabular}
}
\bigskip
\caption{Possibilities for $S_0$ and corresponding dimensions of cohomology groups} \label{788}
\end{table}

In each case, the estimate~\labelcref{est} yields $\hh0.X.{\RR1.f.\O S.}. \le \hh0.X.\can X. + 1$.
In the last case, this is seen as follows, using~\labelcref{RR.2}:
\begin{equation*}
g + \hh0.C.L. - \hh0.C.K_C + L. =
\left\lbrace \hspace{-.4em}
\renewcommand{\arraystretch}{1.2}%
\begin{array}{ll}
g + 1 - g, & \text{$L$ trivial,} \\
g - (g - 1), & \text{$L$ non-trivial,} \\
\end{array}
\hspace{-.4em} \right\rbrace
= 1.
\end{equation*}
Hence if $\hh0.X.\can X. \le 1$, then we can conclude by \cref{lowgenus lz} that $X$ is smooth, just as in the proof of \cref{TX gen nef}.
We will thus from now on assume that $\hh0.X.\can X. = 2$.
In view of the above table, this means that the first line ($S_0 = \PP2$ or ruled) can be excluded.
Also, the algebraic dimension $\mathrm a(S_0) = 1$, as the ratio of two linearly independent sections of $\can X$ provides a non-constant meromorphic function on $X$ and then also on $S_0$.

\begin{clm} \label{771}
Any irreducible curve contained in $S_0$ is smooth elliptic.
\end{clm}

\begin{proof}
Assume first that $S_0$ is a primary Kodaira surface, i.e.~in particular a locally trivial fibration with fibre $F$ an elliptic curve.
If $C \subset S_0$ were a curve not contained in a fibre, then $(C + nF)^2 > 0$ for $n \gg 0$ and so $S_0$ would be projective~\cite[Thm.~IV.6.2]{BHPV04}, which it is not.
Hence \cref{771} is true in this case.
A secondary Kodaira surface admits an \'etale covering by a primary Kodaira surface, so any of its curves must be smooth and then also elliptic by the Hurwitz formula.

We next treat the case where $S_0$ is of class VII\textsubscript0.
By~\cite[Thm.~35]{KodairaStructureII}, $S_0$ is a Hopf surface.
As any Hopf surface has an \'etale covering by a primary Hopf surface~\cite[Thm.~30]{KodairaStructureII}, we may assume by the same argument as above that $S_0$ is itself primary.
Then $S_0$ is the quotient of $W \defn \C^2 \setminus \set0$ by the infinite cyclic group $G$ generated by the automorphism $(z_1, z_2) \mapsto (\alpha_1 z_1, \alpha_2 z_2)$, where $0 < |\alpha_1| \le |\alpha_2| < 1$ and $\alpha_1^k = \alpha_2^\ell$ for certain positive integers $k, \ell$~\cite[Thm.~31]{KodairaStructureII}.
We may assume that $k$ and $\ell$ are minimal with this property.
The non-constant meromorphic function $z_1^k / z_2^\ell$ then defines a map $\phi \from S_0 \to \PP1$ with connected fibres.
We claim that all fibres of $\phi$ are smooth elliptic curves.
To this end, let $\wt\phi \from W \to \PP1$ be the pullback of $\phi$ to the universal covering.
By calculating the differential of $\wt\phi$, we see that this map has rank one at all points $(z_1, z_2) \in W$ with $z_1z_2 \ne 0$.
So all the fibres $F_\lambda \defn \phi\inv(\lambda)$ with $\lambda \ne 0, \infty$ are smooth.
As the second Betti number $b_2(S_0) = 0$, all intersection numbers on $S_0$ are zero.
In particular, $\deg K_{F_\lambda} = \big( K_{S_0} + F_\lambda \big) \cdot F_\lambda = 0$ by adjunction and so $F_\lambda$ is an elliptic curve.

It remains to consider $F_0$ and $F_\infty$.
The fibre $F_0$ is the quotient of $\wt\phi\inv(0) = \set{z_1 = 0} \isom \C^*$ by the group $G$, which acts on $\wt\phi\inv(0)$ via multiplication by $\alpha_2$.
Via the exponential map, $F_0$ is thus seen to be isomorphic to $\C$ modulo a lattice, that is, an elliptic curve.
The argument for $F_\infty$ is the same.
So all fibres of $\phi$ are smooth elliptic.
As in the case of Kodaira surfaces, every curve on $S_0$ is contained in a fibre of $\phi$ and hence the claim is proven in this case, too.

Finally, if $S_0$ is minimally elliptic and $\phi \from S_0 \to C$ is the pluricanonical map, then we have seen that $\deg_C (\phi_* \can{S_0/C}) = 0$.
By~\cite[Thm.~III.18.2]{BHPV04} this implies that the only singular fibres of $\phi$ are multiples of smooth elliptic curves.
In particular, set-theoretically all fibres are smooth elliptic.
Again, there are no other curves except the fibres and so the proof is finished.
\end{proof}

As observed above, we have $\hh0.X.{\RR1.f.\O S.}. \le 3$.
If $X$ has only singularities of genus at most two, we conclude by \cref{lowgenus lz}.
Otherwise, there is a genus~$3$ singularity $x \in X$, and it is the unique singular point of $X$.
If $\Exc(f)$ contains a non-rational curve, then $x \in X$ has $g \ge 1$, hence $p_g - g - b \le 2$ and we are done by \cref{lowgenus lz} again.
If, on the other hand, $\Exc(f)$ consists solely of rational curves, then in particular every $f$-exceptional curve gets contracted to a point by $\pi$ thanks to \cref{771}.
In other words, $\Exc(f) \subset \Exc(\pi)$.
By the Theorem on Formal Functions, $\RR1.f.\O S.$ can be computed as
\[ \varprojlim_Z \HH1.Z.\O Z., \]
where the inverse limit runs over all cycles $Z$ with $\supp Z \subset \Exc(f)$.
By smoothness of $S_0$, we have $\RR1.\pi.\O S. = 0$ and thus, by the Theorem on Formal Functions again, $\HH1.Z.\O Z. = 0$ for all $Z$ with $\supp Z \subset \Exc(\pi)$.
Since $\Exc(f) \subset \Exc(\pi)$, we con\-clude that $x \in X$ is a rational singularity and so $X$ is smooth by \cref{lowgenus lz}. \qed

\begin{rem-plain}
We would like to discuss which parts of the above argument can be generalized, in particular with respect to \cref{788}.
In the first case, $S_0 = \PP2$ or a ruled surface, the condition $\hh0.X.\can X. \le 2$ is automatic by \cref{kahler}, but the existence of two twisted vector fields on $X$ is necessary in \cref{ruled} to exclude ruled surfaces over curves of general type.
If there is only one vector field, all one can say is that such a ruled surface would be decomposable.

The vector fields $v_{1,2}$ are also used in \cref{kahler} to rule out the situation that $S_0$ is \kahler and of non-negative Kodaira dimension.
Without this assumption, several new cases need to be dealt with:
\begin{itemize}
\item If $S_0$ is a complex $2$-torus, then $\hh1.S.\O S. - \hh0.S.\can S. = 1$ and so we can still prove smoothness of $X$ if $\hh0.X.\can X. \le 2$.
Note that \cref{771} does not hold any longer, but this is not a problem because it is still true that $S_0$ contains no rational curves.
\item If $S_0$ is bi-elliptic, then again $\hh1.S.\O S. - \hh0.S.\can S. = 1$ and there are no rational curves on $S_0$.
The conclusion is thus the same as in the torus case.
\item If $S_0$ is a K3 surface, we have $\hh1.S.\O S. - \hh0.S.\can S. = -1$.
Thus the assumption $\hh0.X.\can X. \le 3$ is sufficient.
The case $\hh0.X.\can X. = 4$, however, appears difficult due to the failure of \cref{771}: $S_0$ could certainly contain a lot of rational curves.
\item If $S_0$ is an Enriques surface, then $\hh1.S.\O S. - \hh0.S.\can S. = 0$.
Similarly to the K3 case, $\hh0.X.\can X. \le 2$ is fine, but $\hh0.X.\can X. = 3$ is not.
\item If $\kappa(S_0) = 1$, we use notation as in \cref{kappa 1}.
If $\deg L = 0$, we have already seen that $\hh1.S.\O S. - \hh0.S.\can S. \le 1$.
If $\hh0.X.\can X. \le 1$ or if $g > 0$, the above arguments apply.
The remaining case $\hh0.X.\can X. = 2$ and $g = 0$ needs to be addressed either by showing $\hh0.S.\can S. = 1$ or by excluding any rational curves on $S_0$.

If $\deg L > 0$, then clearly $\hh1.S.\O S. = g + h^0(L \dual) = g$ and $\hh0.S.\can S. \ge h^0(K_C + L) \ge g$, hence the difference is $\le 0$.
The conclusion is as in the Enriques case, because the singular fibres of $\phi$ can very well have rational components.
\item If $S_0$ is of general type, it appears difficult (if not impossible) to give a general upper bound on $\hh1.S.\O S. - \hh0.S.\can S.$ and hence we do not believe that statements in the style of \cref{lowgenus lz} are useful for handling this situation.
\end{itemize}
\end{rem-plain}

\providecommand{\bysame}{\leavevmode\hbox to3em{\hrulefill}\thinspace}
\providecommand{\MR}{\relax\ifhmode\unskip\space\fi MR}
\providecommand{\MRhref}[2]{%
  \href{http://www.ams.org/mathscinet-getitem?mr=#1}{#2}
}
\providecommand{\href}[2]{#2}

\end{document}